\documentclass[english]{article}
\usepackage[T1]{fontenc}
\usepackage[latin1]{inputenc}
\usepackage{graphicx}
\usepackage{amssymb}
\usepackage{amsmath}
\usepackage{amsthm}
\usepackage{color}
\usepackage{relsize}
\usepackage[all]{xy}
\usepackage{amsfonts}
\usepackage{mathrsfs}
\usepackage{latexsym}
\usepackage{geometry} 
\usepackage{textcomp}
\usepackage[english]{babel}

\theoremstyle{plain}
\newtheorem*{bigtheo}{Theorem}
\newtheorem{theo}{Theorem}[section]
\newtheorem{prop}[theo]{Proposition}
\newtheorem{lemm}[theo]{Lemma}

\newtheorem{hypo}[theo]{Hypothesis}

\theoremstyle{definition}
\newtheorem{defi}[theo]{Definition}
\theoremstyle{remark}
\newtheorem{rema}[theo]{Remark}

\newcommand{\T}{\mathcal{T}}
\newcommand{\A}{\mathcal{A}}
\newcommand{\B}{\mathcal{B}}
\newcommand{\C}{\mathcal{C}}
\newcommand{\E}{\mathcal{E}}
\newcommand{\mF}{\mathbb{F}}
\newcommand{\F}{\mathcal{F}}
\newcommand{\Li}{\mathcal{L}}
\newcommand{\wF}{\widetilde{\F}}
\newcommand{\Oo}{\mathcal{O}}

\title{The compatibility with the duality for partial Hasse invariants}
\author{St\'ephane Bijakowski}

\begin{document}

\maketitle

\begin{abstract}
We give a simple and natural proof for the compatibility of the Hasse invariant with duality. We then study a $p$-divisible group with an action of the ring of integers of a finite ramified extension of $\mathbb{Q}_p$. We suppose that it satisfies the Pappas-Rapoport condition ; in that case the Hasse invariant is a product of partial Hasse invariants, each of which can be expressed in terms of primitive Hasse invariants. We then show that the dual of the $p$-divisible group naturally satisfies a Pappas-Rapoport condition, and prove the compatibility with the duality for the partial and primitive Hasse invariants.
\end{abstract}

\tableofcontents

\section*{Introduction}

A $p$-divisible group $G$ over a field of characteristic $p$ is said to be ordinary if it is an extension of a multiplicative part by an \'etale part. To detect the ordinariness, one can use the Hasse invariant $ha(G)$. It is a section of the $p-1$-th power of the sheaf of differentials $\omega_G$, and it is invertible if and only if the $p$-divisible group is ordinary. Since the dual of a multiplicative $p$-divisible group is \'etale, one sees that $G$ is ordinary if and only if its dual $G^D$ is. This suggests that the Hasse invariants of $G$ and $G^D$ are related. Actually, one has the following result.

\begin{bigtheo}
Let $S$ be a scheme of characteristic $p$, and let $G$ be a $p$-divisible group over $S$ of height $h$ and dimension $d$, with $ 0 < d < h$. Then we have an isomorphism $\omega_G^{p-1} \simeq \omega_{G^D}^{p-1}$. The elements $ha(G)$ and $ha(G^D)$ are identified under the induced isomorphism $H^0(S, \omega_G^{p-1}) \simeq H^0(S, \omega_{G^D}^{p-1}) $.
\end{bigtheo}

This theorem has been proved in \cite{Fa}, proposition $2$. In \cite{Co} Th. $2.3.5$, it is proved that the sections $ha(G)$ and $ha(G^D)$ generate the same ideals. However, both proofs are very little natural. Indeed, they first deal with the case where the ordinary locus is dense, and then use a descent argument to prove the general case. \\
We give here a simple and natural proof of this result. For this, we use the first de Rham cohomology group $\E$ of $G$ ; it is a locally free sheaf of rank $h$ over $S$. It is equipped with two filtrations~:~the Hodge filtration $\F$ and the conjugate filtration $\wF$. With these notations, one can show that the section $ha(G)$ is induced by the determinant of the natural map $\F \to \E / \wF$, whereas the section $ha(G^D)$ is induced by the determinant of the map $\wF \to \E / \F$. A simple argument allows us to compare these two sections and thus prove the previous theorem. \\
$ $\\
\indent Suppose now that $G$ has an action of the ring of integers of $K$, where $K$ is a finite extension of $\mathbb{Q}_p$. Suppose for simplicity that $K$ is totally ramified of degree $e$ (the general case is treated in the paper), and let $\pi$ be an uniformizer. We assume that $G$ satisfy the Pappas-Rapoport condition (\cite{P-R}) : there exists a filtration
$$0 = \omega_{G}^{[0]} \subset \omega_{G}^{[1]} \subset \dots \subset \omega_{G}^{[e-1]} \subset \omega_{G}^{[e]} = \omega_{G}$$
such that $\omega_{G}^{[i]} / \omega_{G}^{[i-1]} $ is locally free of rank $d_1$, and $\pi \cdot \omega_{G}^{[i]} \subset \omega_{G}^{[i-1]} $ for all $1 \leq i \leq e$. Here $d_1$ is a fixed integer, and the condition implies that $e d_1 = d$. Define  $\Li_G^{[i]} := \det (\omega_{G}^{[i]}  / \omega_{G}^{[i-1]} )$ for $1 \leq i \leq e$. Then one can show that the Hasse invariant of $G$ is the product of the partial Hasse invariants $ha^{[i]} (G)$ for $1 \leq i \leq e$. The partial Hasse invariant $ha^{[i]} (G)$ is a section of $(\Li_G^{[i]})^{p-1}$ for all integers $i$. Moreover, each of these partial Hasse invariants is a product of the primitive Hasse invariants $hasse(G)$ and $m^{[i]} (G)$ for $2 \leq i \leq e$. The element $m^{[i]} (G)$ is a section of $\Li_G^{[i-1]} (\Li_G^{[i]})^{-1}$, and $hasse(G)$ is a section of $(\Li_G^{[e]})^p (\Li_G^{[1]})^{-1}$. These invariants have been defined in \cite{R-X} for the Hilbert modular variety, and in \cite{Bi} for the general case. We then get the following result.

\begin{bigtheo}
The $p$-divisible group $G^D$ satisfies canonically a Pappas-Rapoport condition. Moreover, the partial Hasse invariants $ha^{[i]} (G)$ for $1 \leq i \leq e$, and the primitive Hasse invariants $hasse(G)$ and $m^{[i]} (G)$ for $2 \leq i \leq e$ are compatible with the duality.
\end{bigtheo}

When we say for example that the section $m^{[i]} (G)$ is compatible with the duality, it means that there is a isomorphism between the sheaves $\Li_G^{[i-1]} (\Li_{G}^{[i]})^{-1}$ and $\Li_{G^D}^{[i-1]} (\Li_{G^D}^{[i]})^{-1}$, and that the two sections $m^{[i]} (G)$ and $m^{[i]} (G^D)$ are equal under the induced isomorphism. \\
The result for the sections $m^{[i]} (G)$ is proved unconditionally, but for the sections $hasse(G)$, we need to use crystalline Dieudonn\'e theory at some point, so we assume that $S$ is either smooth over $\mF_p$ or locally the spectrum of a semi-perfect ring. Note that for the special case of a $p$-divisible group defined over the ring of integers of an extension of $\mathbb{Q}_p$, the fact that the valuations of $m^{[i]} (G)$ and $m^{[i]} (G^D)$ are the same has been proved by explicit computations in \cite{Bi}. \\
$ $\\
Let us talk briefly about the organization of the paper. In the first section, a new simple proof for the compatibility of the Hasse invariant with duality is given. The second part deals with the case of an unramified action. Finally, in the third section we study the general case for $p$-divisible groups satisfying the Pappas-Rapoport condition. \\
$ $\\
The author would like to thank Valentin Hernandez and James Newton for interesting discussions.

\section{The Hasse invariant}

Let $S$ be a $\mF_p$-scheme, and $G$ be a $p$-divisible over $S$ of height $h$ and dimension $d$, with $0 < d <h$. Let us denote $\E = H_{dR}^1 (G)$ ; it is a locally free sheaf over $S$ of rank $h$. It is also the evaluation of the contravariant Dieudonn\'e crystal on $S$ (see \cite{BBM} section $3.3$). The Frobenius and the Verschiebung induce maps $F : \E^{(p)} \to \E$ and $V : \E \to \E^{(p)}$, where $\E^{(p)}$ denotes the twist of $\E$ by the Frobenius. The Hodge filtration gives a subsheaf $\F \subset \E$, locally free of rank $d$, and which induces the exact sequence (see \cite{BBM} corollary $3.3.5$)
$$0 \to \omega_G \to \E \to \omega_{G^D}^\vee \to 0$$
where $\omega_G$ is the conormal sheaf of $G$ along its unit section, $G^D$ is the Cartier dual of $G$, and the notation $\mathcal{G}^\vee$ means the dual of the sheaf $\mathcal{G}$. Moreover, $\F^{(p)} = $ Im $V = $ Ker $F$ (see \cite{EV} section $3.1$). Let $\widetilde{\F}$ denote the subsheaf of $\E$ defined by Ker $V = $ Im $F$ ; we will call $\wF$ the conjugate filtration. The Frobenius and Verschiebung induce isomorphisms

\begin{displaymath}
F : (\E / \F)^{(p)} \simeq \widetilde{\F}      \qquad  \qquad   V : \E / \widetilde{\F} \simeq \F^{(p)}
\end{displaymath}

\noindent The sheaf $\wF$ is thus locally free of rank $h - d$ and induces an exact sequence
$$0 \to (\omega_{G^D}^\vee)^{(p)} \to \E \to \omega_G^{(p)} \to 0$$

\begin{defi}
The Hasse map $Ha(G)$ for $G$ is the map $V : \omega_G \to \omega_G^{(p)}$. Let $\Li_G$ be the invertible sheaf defined by $\det \omega_G$. The determinant of the Hasse map induces a section $ha(G) \in H^0(S, \Li_G^{p-1})$, called the Hasse invariant of $G$.
\end{defi}

\noindent The sheaf $H_{dR}^1 (G^D)$ is $\E^\vee$, and the Hodge filtration on this space is induced by $\F^\bot$ (see \cite{BBM} section $5.3$). Moreover, the Frobenius and Verschiebung are given respectively by
$$V^\vee : (\E^\vee)^{(p)} \to \E^\vee \qquad \qquad F^\vee : \E^\vee \to (\E^\vee)^{(p)}$$
The map $Ha(G^D)^\vee$ is thus the map $F : (\E / \F)^{(p)} \to \E / \F$. 

\begin{lemm}
Using the previous isomorphisms, the Hasse map $Ha(G)$ is the natural map $\F \to~\E / \wF$, obtained by composing the inclusion of $\F$ in $\E$ with the projection to $\E / \wF$. Similarly, the map $Ha(G^D)^\vee$ is the natural map $\wF \to \E / \F$.
\end{lemm}

\begin{proof}
The Hasse map is by definition the map induced by the Verschiebung $\F \to \F^{(p)}$. Since the inverse of the Verschiebung gives an isomorphism $\F^{(p)} \simeq \E / \widetilde{\F}$, the claim follows. The map $Ha(G^D)^\vee$ is the map $F : (\E / \F)^{(p)} \to \E / \F$. The composition of the inverse of the Frobenius $\wF \simeq (\E / \F)^{(p)}$ with this map gives the natural map $\wF \to \E / \F$.
\end{proof}

Before proving the duality compatibility for the Hasse invariant, let us state a general proposition, which will be useful throughout the paper.

\begin{prop} \label{dual}
Let $\mathcal{A}$ be a locally free sheaf of rank $r$ over $S$, and let $0 < s < r$ be an integer. Let $\B \subset \A$ and $\C \subset \A$ be two locally free sheaves of rank respectively $s$ and $r-s$, such that $\A / \B$ and $\A/ \C$ are locally free. Then we have an isomorphism of invertible sheaves
$$\det(\A / \B) \otimes \det(\C)^{-1} \simeq \det(\A / \C) \otimes \det(\B)^{-1}$$
Let $x \in H^0(S, \det(\A / \B) \otimes \det(\C)^{-1})$ be the section corresponding to the determinant of the natural map $\C \to \A / \B$. Then $x$ is mapped to $y$ under the isomorphism
$$H^0(S, \det(\A / \B) \otimes \det(\C)^{-1}) \simeq H^0(S, \det(\A / \C) \otimes \det(\B)^{-1})$$
where $y \in H^0(S, \det(\A / \C) \otimes \det(\B)^{-1})$ is the section corresponding to the determinant of the natural map $\B \to \A / \C$.
\end{prop}

\begin{proof}
We have isomorphisms of invertible sheaves
$$\det(\A) \simeq \det(\A / \B) \otimes \det(\B) \simeq \det(\A / \C) \otimes \det(\C)$$
so the invertible sheaves $\det(\A / \B) \otimes \det(\C)^{-1}$ and $\det(\A / \C) \otimes \det(\B)^{-1}$ are both isomorphic to 
$$\det(\A) \otimes \det(\B)^{-1} \otimes \det(\C)^{-1}$$
We thus have isomorphisms
$$H^0(S, \det(\A / \B) \otimes \det(\C)^{-1}) \simeq H^0(S, \det(\A) \otimes \det(\B)^{-1} \otimes \det(\C)^{-1}) \simeq H^0(S, \det(\A / \C) \otimes \det(\B)^{-1})$$
One then sees that the elements $x$ and $y$ are mapped to the same element in $H^0(S, \det(\A) \otimes~\det(\B)^{-1} \otimes~\det(\C)^{-1})$. Namely, they are mapped to the section induced by the determinant of the map
$$\B \oplus \C \to \A$$
\end{proof}

\noindent By a slight abuse of notation, we will say that $x=y$ under the isomorphism of invertible sheaves $\det(\A / \B) \otimes \det(\C)^{-1} \simeq~\det(\A / \C) \otimes \det(\B)^{-1}$. This proposition allows us to give a simple proof for the compatibility with the duality of the Hasse invariant, obtained in \cite{Fa} proposition $2$.

\begin{theo}
There is an isomorphism $\Li_G^{p-1} \simeq \Li_{G^D}^{p-1}$. With this isomorphism, one has the equality $ha(G) = ha(G^D)$.
\end{theo}

\begin{proof}
We apply the previous proposition to $\A = \E$, $\B = \F$ and $\C = \wF$. Since $(\det \E / \wF) \otimes~(\det \F)^{-1} \simeq~\Li_G^{p-1}$ and $(\det \E / \F) \otimes (\det \wF)^{-1} \simeq \Li_{G^D}^{p-1}$, the first result follows. \\
From the previous lemma, $ha(G)$ is the section of the invertible sheaf $(\det \E / \wF) \otimes (\det \F)^{-1}$ obtained by taking the determinant of the natural map $\F \to \E / \wF$. Similarly, $ha(G^D)^\vee$ is the section of the invertible sheaf $(\det \E / \F) \otimes (\det \wF)^{-1}$ obtained by taking the determinant of the natural map $\wF \to \E / \F$. Note that $ha(G^D)^\vee$ and $ha(G^D)$ induces the same section under the canonical isomorphism of sheaves
$$(\det \E / \F) \otimes (\det \wF)^{-1} \simeq  (\det \wF^\vee) \otimes (\det (\E / \F)^\vee)^{-1}$$
The second part of the proposition allows us to conclude.
\end{proof}

\section{Unramified partial Hasse invariants}

Let $K_0$ be a finite unramified extension of degree $f \geq 1$ of $\mathbb{Q}_p$, and let $O_{K_0}$ be its ring of integers, and $k$ the residue field. We suppose that $S$ is a $k$-scheme, and that the $p$-divisible group $G$ has an action of $O_{K_0}$. Let $\T$ be the set of embeddings $K_0 \to \overline{\mathbb{Q}_p}$ ; then $\T$ is a cyclic group of order $f$ generated by the Frobenius $\sigma$. The sheaf $\E$ is then free over $\Oo_S \otimes_{\mathbb{Z}_p} O_{K_0}$ (\cite{FGL} section I.B.1), and decomposes into $\E = \oplus_{i=1}^f \E_i$, with $O_{K_0}$ acting on $\E_i$ by $\sigma^i$ for all $1 \leq i \leq f$. Each $\E_i$ is locally free of rank $h_0$, with $f h_0 = h$. The Frobenius and Verschiebung induce maps
$$F_i : \E_{i-1}^{(p)} \to \E_i  \qquad  \qquad    V_i : \E_i \to \E_{i-1}^{(p)} $$
where $\E_0$ is identified with $\E_f$ (we will always identify the indexes $0$ and $f$). For each integer $i$ between $1$ and $f$, define $\F_i := \F \cap \E_i$ and $\wF_i := \wF \cap \E_i$. We make the following assumption.

\begin{hypo}
There exists an integer $0 < d_0 < h_0$ such that $\F_i$ is locally free of rank $d_0$ for all $1 \leq i \leq f$.
\end{hypo} 

\noindent This hypothesis means that there is no obstruction for $G$ to be ordinary. If it is not satisfied, then the Hasse invariant of $G$ is always $0$. \\
For all $1 \leq i \leq f$, we have $\F_{i-1}^{(p)} = $ Im $V_i = $ Ker $F_i$, and $\wF_i = $~Ker~$V_i = $~Im~$F_i$. Moreover, we have isomorphisms

\begin{displaymath}
F_i : (\E_{i-1} / \F_{i-1})^{(p)} \simeq \wF_i      \qquad  \qquad   V_i : \E_i / \widetilde{\F}_i \simeq \F_{i-1}^{(p)}
\end{displaymath}
The sheaves $\wF_i$ are thus locally free of rank $h - d_0$. We have a decomposition $\omega_G = \oplus_{i=1}^f \omega_{G,i}$ with $O_{K_0}$ acting on $\omega_{G,i}$ by $\sigma^i$ for $1 \leq i \leq f$. The filtrations induced by $\F_i$ and $\wF_i$ give respectively the exact sequences

\begin{align*}
0  \to \omega_{G,i} & \to  \E_i   \to \omega_{G^D,i}^\vee  \to 0 \\
0  \to (\omega_{G^D,i-1}^\vee)^{(p)} & \to  \E_i  \to \omega_{G,i-1}^{(p)}  \to 0
\end{align*}

\noindent Define the invertible sheaf $\Li_{G,i} := \det \omega_{G,i}$ for all integers $i$. Let us now fix an integer $1 \leq i \leq f$.

\begin{defi}
The partial Hasse map $Ha_i(G)$ is the map $V_i : \omega_{G,i} \to \omega_{G,i-1}^{(p)}$. It induces a section called the partial Hasse invariant $ha_i(G) \in H^0 ( S, \Li_{G,i-1}^p \Li_{G,i}^{-1})$.
\end{defi}

The map $Ha_i(G^D)^\vee$ is the map $F_i : (\E_{i-1} / \F_{i-1})^{(p)} \to \E_i / \F_i$. The proof of the following lemma is analogous to the one from the previous section.

\begin{lemm}
The Hasse map $Ha_i(G)$ is the natural map $\F_i \to \E_i / \wF_i$, obtained by composing the inclusion of $\F_i$ in $\E_i$ with the projection to $\E_i / \wF_i$. Similarly, the map $Ha_i(G^D)^\vee$ is the natural map $\wF_i \to \E_i / \F_i$.
\end{lemm}

\begin{theo}
There is an isomorphism $\Li_{G,i-1}^{p} \Li_{G,i}^{-1} \simeq \Li_{G^D,i-1}^{p} \Li_{G^D,i}^{-1}$. Using this isomorphism, we have $ha_i(G) = ha_i(G^D)$.
\end{theo}

\begin{proof}
We apply the proposition \ref{dual} to $\A = \E_i$, $\B = \F_i$ and $\C = \wF_i$. Note that we have
$$\det( \F_i) = \Li_{G,i} \qquad \det( \E_i / \F_i) = \Li_{G^D,i}^{-1} \qquad \det( \wF_i) = \Li_{G^D,i-1}^{-p} \qquad \det( \E_i / \wF_i) = \Li_{G,i-1}^p$$
\end{proof}

\section{Hasse invariants for $p$-divisible groups with the Pappas-Rapoport condition}

\subsection{The primitive Hasse invariants $m_i^{[j]}$}

Let $K$ be a totally ramified extension of $K_0$ of degree $e \geq 2$, $O_K$ its ring of integers, and $\pi$ an uniformizer. Suppose that the $p$-divisible group $G$ has an action of $O_K$. Since $\E$ is locally free over $\Oo_S \otimes_{\mF_p} O_K/p$ (\cite{FGL} section I.B.1), each $\E_i$ is locally free over $\Oo_S \otimes_{k,\sigma^i} O_K/p$ of rank $h_1$. We have an isomorphism $O_K/p \simeq k[X]/ X^e$, with $\pi$ mapping to $X$ ; we then see that $e h_1 = h_0$. \\
We will denote by $\E_i [\pi^j]$ the kernel of the multiplication by $\pi^j$ on $\E_i$, for all $1 \leq j \leq e-1$ and $1 \leq i \leq f$. The multiplication by $\pi^j$ induces an isomorphism
$$\E_i / \E_i[\pi^j] \simeq \E_i[\pi^{e-j}]$$
The inverse of this map will be called the division by $\pi^j$ ; it is defined on $\E_i[\pi^{e-j}]$. Let $0 < d_1 < h_1$ be an integer, and we make the following assumption.

\begin{hypo}
There exists a filtration $0 = \omega_{G,i}^{[0]} \subset \omega_{G,i}^{[1]} \subset \dots \subset \omega_{G,i}^{[e-1]} \subset \omega_{G,i}^{[e]} = \omega_{G,i}$ for each $1 \leq i \leq f$, such that
\begin{itemize}
\item each $\omega_{G,i}^{[j]}$ is locally a direct summand of $\omega_{G,i}$
\item $\omega_{G,i}^{[j]} / \omega_{G,i}^{[j-1]}$ is locally free of rank $d_1$ for all $1 \leq j \leq e$.
\item $\pi \cdot \omega_{G,i}^{[j]} \subset \omega_{G,i}^{[j-1]}$ for all $1 \leq j \leq e$.
\end{itemize}
\end{hypo}

\begin{rema}
This condition is the Pappas-Rapoport condition (see \cite{P-R}). In general, the rank of the graded parts $\omega_{G,i}^{[j]} / \omega_{G,i}^{[j-1]}$ are fixed but may depend on $i$ and $j$. But if this rank is not constant, then the Hasse invariant is $0$. 
\end{rema}

\noindent Let us fix an integer $1 \leq i \leq f$ until the rest of the paper. Since the sheaf $\F_i$ is isomorphic to $\omega_{G,i}$, we have a filtration $(\F_i^{[j]})_{0 \leq j \leq e}$ on $\F_i$.

\begin{prop}
The filtration on $\omega_{G,i}$ induces canonically a filtration on $\omega_{G^D,i}$, satisfying the same properties as in the hypothesis, with $d_1$ replaced by $h - d_1$.
\end{prop}

\begin{proof}
For $1 \leq j \leq e$, we define
$$\F_i^{[e+j]} := (\pi^j)^{-1} \F_i^{[e-j]} $$
Since $\F_i^{[e-j]}$ is killed by $\pi^{e-j}$, it lies in $\E_i[\pi^{e-j}]$, and thus $\F_i^{[e+j]}$ is locally free of rank $j h_1 +~(e-~j) d_1$. The sheaf $\F_i^{[e+j]}$ is a direct summand of $\E_i$, and since $\pi \cdot \F_i^{[e-j]} \subset \F_i^{[e-(j+1)]}$, one has $\F_i^{[e+j]} \subset~\F_i^{[e+j+1]}$. One has then a filtration
$$\F_i^{[e]} \subset \F_i^{[e+1]} \subset \dots \subset \F_i^{[2e-1]} \subset \F_i^{[2e]} = \E_i$$
and $\F_i^{[e+j]} / \F_i^{[e+j-1]}$ is locally free of rank $h_1 - d_1$ for $1 \leq j \leq e$. Moreover, we deduce from the inclusion $\F_i^{[e-j]} \subset \F_i^{[e-(j-1)]}$ that $\pi \cdot \F_i^{[e+j]} \subset \F_i^{[e+j-1]}$ for all $1 \leq j \leq e$. \\
Since $\E_i / \F_i^{[e]}$ is isomorphic to $\omega_{G^D,i}^\vee$, we have thus constructed a filtration on the sheaf $\omega_{G^D,i}$ satisfying the required properties.
\end{proof}

\noindent We then have a filtration
$$0 = \omega_{G^D,i}^{[0]} \subset \omega_{G^D,i}^{[1]} \subset \dots \subset \omega_{G^D,i}^{[e-1]} \subset \omega_{G^D,i}^{[e]} = \omega_{G^D,i}$$
with the property that $\omega_{G^D,i}^{[j]} \simeq (\F_i / \F_i^{[2e-j]})^\vee$ and $\omega_{G^D,i}^{[j]} / \omega_{G^D,i}^{[j-1]} \simeq (\F_i^{[2e-j+1]} / \F_i^{[2e-j]})^\vee$ for all $j$ between $1$ and $e$. \\
Define $\Li_{G,i}^{[j]} := \det (\F_i^{[j]} / \F_i^{[j-1]})$, and $\Li_{G^D,i}^{[j]} := \det (\F_i^{[2e-j+1]} / \F_i^{[2e-j]})^\vee$ for $1 \leq j \leq e$.

\begin{defi}
Let $2 \leq j \leq e$ be an integer. We define the primitive Hasse maps $M_{i}^{[j]}(G)$ as the map $\omega_{G,i}^{[j]} / \omega_{G,i}^{[j-1]} \to \omega_{G,i}^{[j-1]} / \omega_{G,i}^{[j-2]}$ induced by the multiplication by $\pi$. The determinants of these maps give the primitive Hasse invariants $m_{i}^{[j]} (G) \in H^0 \left(S,\Li_{G,i}^{[j-1]} (\Li_{G,i}^{[j]})^{-1}\right)$.
\end{defi}

\noindent These primitive Hasse invariants have been defined in \cite{R-X} for the Hilbert modular variety (see also \cite{Bi}). The dual of the map $M_{i}^{[j]}(G^D)$ is thus the map $\F_i^{[2e+2-j]} / \F_i^{[2e+1-j]} \to~\F_i^{[2e+1-j]} / \F_i^{[2e-j]}$ induced by the multiplication by $\pi$ for all $2 \leq j \leq e$. \\
Let us define ${\F_i^{[j]}}' := \pi^{-1} \left(\F_i^{[j]} \right)$ for all $0 \leq j \leq e-1$. It is a locally free sheaf of rank $h_1 + j d_1$.

\begin{lemm}
The multiplication by $\pi$ induces an isomorphism ${\F_i^{[j-1]}}' / {\F_i^{[j-2]}}' \simeq \F_i^{[j-1]} / \F_i^{[j-2]}$ for all $2 \leq j \leq e$. Using this isomorphism, the map $M_{i}^{[j]}(G)$ is the natural map $\F_i^{[j]} / \F_i^{[j-1]} \to~{\F_i^{[j-1]}}' / {\F_i^{[j-2]}}'$. 
\end{lemm}

\begin{proof}
Let $j$ be an integer between $2$ and $e$. Since $\F_i^{[j-1]}$ lies in $\E_i[\pi^{e-1}]$, the multiplication by $\pi$ gives a surjective map
$${\F_i^{[j-1]}}' \to {\F_i^{[j-1]}} / {\F_i^{[j-2]}}$$
The kernel is exactly ${\F_i^{[j-2]}}'$, so the multiplication by $\pi$ induces an isomorphism 
$${\F_i^{[j-1]}}' / {\F_i^{[j-2]}}' \simeq \F_i^{[j-1]} / \F_i^{[j-2]}$$
The composition of $M_{i}^{[j]}(G)$ with the inverse of this isomorphism gives the natural map
$$\F_i^{[j]} / \F_i^{[j-1]} \to {\F_i^{[j-1]}}' / {\F_i^{[j-2]}}'$$
\end{proof}

\begin{lemm}
The multiplication by $\pi^{e-j}$ induces an isomorphism $\F_i^{[2e+1-j]} / \F_i^{[2e-j]} \simeq {\F_i^{[j-1]}}' / \F_i^{[j]}$ for all $1 \leq j \leq e$. Using these isomorphisms, the map ${M_{i}^{[j]}(G^D)}^\vee$ is the natural map
$${\F_i^{[j-2]}}' / \F_i^{[j-1]} \to {\F_i^{[j-1]}}' / \F_i^{[j]}$$
for $2 \leq j \leq e$.
\end{lemm}

\begin{proof}
Let $j$ be an integer between $1$ and $e$ ; we have $\F_i^{[2e+1-j]} = (\pi^{e+1-j})^{-1} \F_i^{[j-1]} = (\pi^{e-j})^{-1} {\F_i^{[j-1]}}'$. Since ${\F_i^{[j-1]}}'$ is killed by $\pi^j$, it lies in $\E_i [\pi^j]$. The multiplication by $\pi^{e-j}$ induces a surjective map
$$\F_i^{[2e+1-j]} \to {\F_i^{[j-1]}}' / \F_i^{[j]}$$
The kernel of this map is $(\pi^{e-j})^{-1} \F_i^{[j]} = \F_i^{[2e-j]}$. The multiplication by $\pi^{e-j}$ thus gives an isomorphism
$$\alpha_i^{[j]} : \F_i^{[2e+1-j]} / \F_i^{[2e-j]} \simeq {\F_i^{[j-1]}}' / \F_i^{[j]}$$
Now suppose $2 \leq j \leq e$, and consider the diagram
\begin{displaymath}
\xymatrix@C+40pt{
\F_i^{[2e+2-j]} / \F_i^{[2e+1-j]}   \ar[d]_{\alpha_i^{[j-1]}}   \ar[r]^{{M_{i}^{[j]}(G^D)}^\vee}      &            \F_i^{[2e+1-j]} / \F_i^{[2e-j]} \ar[d]^{\alpha_i^{[j]}} \\
{\F_i^{[j-2]}}' / \F_i^{[j-1]}            & {\F_i^{[j-1]}}' / \F_i^{[j]} 
}
\end{displaymath}
Since both $\alpha_i^{[j-1]}$ and $\alpha_i^{[j]}$ are isomorphisms, ${M_{i}^{[j]}(G^D)}^\vee$ induce a map
$${\F_i^{[j-2]}}' / \F_i^{[j-1]}  \to {\F_i^{[j-1]}}' / \F_i^{[j]} $$
This map is the composition of the division by $\pi^{e-j+1}$, the multiplication by $\pi$, and the multiplication by $\pi^{e-j}$, hence is the natural map induced by the inclusion ${\F_i^{[j-2]}}' \subset {\F_i^{[j-1]}}'$.
\end{proof}

\begin{theo}
Let $2 \leq j \leq e$ be an integer. There is an isomorphism $\Li_{G,i}^{[j-1]} (\Li_{G,i}^{[j]})^{-1} \simeq \Li_{G^D,i}^{[j-1]} (\Li_{G^D,i}^{[j]})^{-1}$. With this isomorphism, one has the equality $m_i^{[j]}(G) =m_i^{[j]}(G^D)$.
\end{theo}

\begin{proof}
We apply the proposition \ref{dual} to $\A = {\F_i^{[j-1]}}' / \F_i^{[j-1]}$, $\B = \F_i^{[j]} / \F_i^{[j-1]}$ and $\C = {\F_i^{[j-2]}}' / \F_i^{[j-1]}$.

\end{proof}

\subsection{The primitive Hasse invariants $hasse_i$}

We suppose until the rest of the paper that $S$ is either smooth over $k$, or locally the spectrum of a semi-perfect ring (i.e. a ring on which the Frobenius is surjective). \\
Let $1 \leq i \leq f$ be a fixed integer ; recall that the sheaf $\F_i^{[1]}$ is killed by $\pi$, and that the division by $\pi^{e-1}$ induces an isomorphism $\E_i [\pi] \to \E_i / \E_i[\pi^{e-1}]$. We will now define the remaining primitive Hasse invariants (see \cite{R-X} and \cite{Bi}).

\begin{defi}
We define the map $Hasse_i(G)$ as the map
\begin{displaymath}
\xymatrix{
\F_i^{[1]} \ar[r] & \E_i[\pi] \simeq \E_i / \E_i[\pi^{e-1}]  \ar[r]^-{V_i} & (\F_{i-1}^{[e]} / \F_{i-1}^{[e-1]})^{(p)}
}
\end{displaymath}
\end{defi}

\noindent The map $Hasse_i(G)$ is then the composition of the division by $\pi^{e-1}$ and the Verschiebung. The map $Hasse_i(G^D)^\vee$ is the composition
\begin{displaymath}
\xymatrix{
(\F_{i-1}^{[e+1]} / \F_{i-1}^{[e]})^{(p)} \ar[r]^-{F_i} & \E_i[\pi] \simeq \E_i / \E_i[\pi^{e-1}]  \ar[r] & \E_i / \F_{i}^{[2e-1]}
}
\end{displaymath}

\noindent The determinant of the map $Hasse_i(G)$ induces the primitive Hasse invariant $hasse_i(G) \in H^0(S, (\Li_{G,i-1}^{[e]})^p (\Li_{G,i}^{[1]})^{-1})$. Similarly, the determinant of the map $Hasse_i(G^D)^\vee$ induces a section $hasse_i(G^D) \in H^0(S, (\Li_{G^D,i-1}^{[e]})^p (\Li_{G^D,i}^{[1]})^{-1})$.

\begin{defi}
We define $\wF_i^{[e+j]} := V_i^{-1} ((\F_{i-1}^{[j]})^{(p)})$ and $\wF_i^{[j]} := F_i ((\F_{i-1}^{[e+j]})^{(p)})$ for all $0 \leq j \leq e$.
\end{defi}

\noindent We thus have a filtration
$$0 = \wF_i^{[0]} \subset \wF_i^{[1]} \subset \dots \subset \wF_i^{[2e-1]} \subset \wF_i^{[2e]} = \E_i$$
with $\wF_i^{[e]} = \wF_i$. One also easily checks that $\pi \cdot \wF_i^{[j]} \subset \wF_i^{[j-1]}$ for $1 \leq j \leq 2e$. The Frobenius and Verschiebung induce isomorphisms

\begin{displaymath}
F_i : (\F_{i-1}^{[e+j]} / \F_{i-1}^{[e+j-1]})^{(p)} \simeq  \wF_i^{[j]} / \wF_i^{[j-1]}     \qquad  \qquad   V_i : \wF_i^{[e+j]} / \wF_i^{[e+j-1]} \simeq (\F_{i-1}^{[j]} / \F_{i-1}^{[j-1]})^{(p)}
\end{displaymath}

\noindent for all $1 \leq j \leq e$. The sheaves $\wF_i^{[j]}$ and $\wF_i^{[e+j]}$ are then locally free of rank respectively $(h_1 - d_1)j$ and $e(h_1 - d_1) + j d_1$ for $0 \leq j \leq e$. Let $u=p / \pi^e \in O_K^\times$.

\begin{lemm}
Let $1 \leq j \leq e-1$ be an integer. The map
\begin{displaymath}
\xymatrix{
F_i^{-1} (\E_{i} [\pi^{e-j}]) \ar[r]^-{F_i} & \E_{i} [\pi^{e-j}] \simeq \E_i / \E_{i} [\pi^{j}] \ar[r]^-{V_i} & (\E_{i-1} / (\pi^{e-j} \cdot \F_{i-1})  )^{(p)}
}
\end{displaymath}
is the multiplication by $\pi^{e-j} u$.
\end{lemm}

\begin{proof}
One can work locally on $S$, so let us suppose that $S = $ Spec $R$ is affine, where $R$ is a ring of characteristic $p$. Let $f$ be the map of the lemma, and let $x \in F_i^{-1} (\E_{i} [\pi^{e-j}])$. \\
If $S$ is locally the spectrum of a semi-perfect ring, then one can suppose that $R$ is semi-perfect. Using $p$-adic Hodge theory, one can then construct a ring $A_{cris}(R)$, which is a $\mathbb{Z}_p$-algebra (see for example \cite{SW} proposition $4.1.3$). It is also equipped with a $PD$ structure compatible with the natural one on $\mathbb{Z}_p$. The prime $p$ is not a zero divisor in $A_{cris}(R)$, and there is a surjective reduction map $A_{cris}(R) \to R$. The Dieudonn\'e module of $G$ gives a free $A_{cris}(R)$-module $M$ such that $M \otimes_{A_{cris}(R)} R \simeq \E$. Let $M^{(p)}$ be the twist of $M$ by the Frobenius on $A_{cris}(R)$ ; the module $M$ is equipped with a Frobenius $F : M^{(p)} \to M$ and a Verschiebung $V : M \to M^{(p)}$, which are linear maps compatible with the maps on $\E$ and satisfy $V \circ F = p$. Moreover, the module $M$ is equipped with an action of $O_K$, so is a $A_{cris} \otimes_{\mathbb{Z}_p} O_K$-module. We refer to \cite{SW} section $4.1$ for more details on the Dieudonn\'e module. Let $x_0 \in M$ be a lift of $x$. Then a lift of $f(x)$ is $V(F(x)) / \pi^j$. Since $V \circ F = p = \pi^e u$, then $\pi^{e-j} u \cdot x_0$ is a lift of $f(x)$. \\
If $S$ is smooth over $k$, then one can assume that the ring $R$ admits a lift $R_0$ which is smooth over the ring of Witt vectors of $k$. Since the reduction map $R_0 \to R$ is $p$-adic, the ring $R_0$ has a canonical PD structure. One can then evaluate the Dieudonn\'e crystal of $G$ on the ring $R_0$ (\cite{BBM} section $3.3$), and the same argument as before gives the result.
\end{proof}

\begin{prop}
We have $\pi^j \cdot \wF_i^{[e+j]} = \wF_i^{[e-j]}$, for all $1 \leq j \leq e-1$.
\end{prop}

\begin{proof}
Let $U$ be any open of $S$, and let $x \in \wF_i^{[e-j]}(U)$. Then there exists $y \in (\F_{i-1}^{[2e-j]})^{(p)} (U)$ such that $x = F_i(y)$. Since $x$ is killed by $\pi^{e-j}$, there exists $z \in \E_i (U)$ such that $x = \pi^j \cdot z$. The previous lemma implies that $\pi^{e-j} u \cdot y = V_i(z)$ in $(\E_{i-1} / (\pi^{e-j} \cdot \F_{i-1})  )^{(p)} (U)$. But 
$$\pi^{e-j} u \cdot~y \in~\pi^{e-j} \cdot~(\F_{i-1}^{[2e-j]})^{(p)} (U) = (\F_{i-1}^{[j]})^{(p)} (U)$$
and $(\pi^{e-j} \cdot \F_{i-1})^{(p)} (U) \subset (\F_{i-1}^{[j]})^{(p)} (U)$. Thus $V_i(z) \in (\F_{i-1}^{[j]})^{(p)} (U)$ and $z \in \wF_i^{[e+j]} (U)$. This proves the inclusion
$$\wF_i^{[e-j]} \subset \pi^j \cdot \wF_i^{[e+j]} $$ 
Suppose now that $x \in \wF_i^{[e+j]} (U)$. This means that $V_i(x) \in (\F_{i-1}^{[j]})^{(p)} (U) = \pi^{e-j} \cdot (\F_{i-1}^{[2e-j]})^{(p)} (U) $. There exists $y \in  (\F_{i-1}^{[2e-j]})^{(p)} (U)$ such that $V_i(x) = \pi^{e-j} u y$. The element $y$ is in $F_i^{-1} (\E_{i} [\pi^{e-j}]) (U)$, and it follows from the previous proposition that the elements $F_i(y)$ and $\pi^j x$ have the same image by the morphism
$$\E_{i} [\pi^{e-j}] (U) \simeq \E_i / \E_{i} [\pi^{j}] (U) \overset{V_i}{\longrightarrow} (\E_{i-1} / (\pi^{e-j} \cdot \F_{i-1})  )^{(p)} (U)$$
The kernel of this morphism is $\pi^j ( \E_i[\pi^j] + \wF_i) (U) \subset \wF_i^{[e-j]} (U)$. Since $F_i(y)$ is also in $\wF_i^{[e-j]} (U)$, this implies that $\pi^j x \in \wF_i^{[e-j]} (U)$, proving the inclusion in the other direction.
\end{proof}

Using this relation for the flag $(\wF_i^{\bullet})$, we get the following description for the maps $Hasse_i(G)$ and $Hasse_i(G^D)^\vee$.

\begin{prop}
There is an isomorphism $(\F_{i-1}^{[e]} / \F_{i-1}^{[e-1]})^{(p)} \simeq  \E_i[\pi] / \wF_i^{[1]}$. Using this isomorphism, the map $Hasse_i(G)$ is the natural map
$$\F_i^{[1]} \to \E_i[\pi] / \wF_i^{[1]}$$
\end{prop}

\begin{proof}
The Verschiebung induce an isomorphism
$$\E_i / \wF_i^{[2e-1]} \simeq (\F_{i-1}^{[e]} / \F_{i-1}^{[e-1]})^{(p)}$$
The multiplication by $\pi^{e-1}$ induces an isomorphism
$$\E_i / \wF_i^{[2e-1]} \simeq \E_i[\pi] / \wF_i^{[1]}$$
The map $Hasse_i(G)$ then induces a map $\F_i^{[1]} \to \E_i[\pi] / \wF_i^{[1]}$, which is the composition of the division by $\pi^{e-1}$, the Verschiebung, the inverse of the Verschiebung, and the multiplication by $\pi^{e-1}$. We thus get the natural map.
\end{proof}

\begin{prop}
There are isomorphisms $(\F_{i-1}^{[e+1]} / \F_{i-1}^{[e]})^{(p)} \simeq \wF_i^{[1]}$, and $\E_i / \F_{i}^{[2e-1]} \simeq \E_i[\pi] / \F_i^{[1]}$. Using these isomorphisms, the map $Hasse_i(G^D)^\vee$ is the natural map 
$$\wF_i^{[1]} \to \E_i[\pi] / \F_i^{[1]}$$
\end{prop}

\begin{proof}
The first isomorphism is induced by the Frobenius, and the second by the multiplication by $\pi^{e-1}$. Using these isomorphisms, the map $Hasse_i(G^D)^\vee$ induces a map
$$\wF_i^{[1]} \to \E_i[\pi] / \F_i^{[1]}$$
which is by definition the composition of the inverse of the Frobenius, the Frobenius, the division by $\pi^{e-1}$, and the multiplication by $\pi^{e-1}$. This is then the natural map.
\end{proof}

\begin{theo}
There is an isomorphism $(\Li_{G,i-1}^{[e]})^p (\Li_{G,i}^{[1]})^{-1} \simeq (\Li_{G^D,i-1}^{[e]})^p (\Li_{G^D,i}^{[1]})^{-1}$. Moreover, we have $hasse_i(G) = hasse_i(G^D)$ using the previous isomorphism.
\end{theo}

\begin{proof}
We apply the proposition \ref{dual} to $\A = \E_i[\pi]$, $\B = \F_i^{[1]}$ and $\C = \wF_i^{[1]}$.

\end{proof}

\subsection{Partial Hasse invariants}

The Verschiebung respects the filtration $\F_{\bullet}^{[\bullet]}$. More precisely, we have the following result.

\begin{prop}
Let $1 \leq i \leq f$, and $1 \leq j \leq e$. Then $V_i$ induces a map
$$\F_i^{[j]} \to (\F_{i-1}^{[j]})^{(p)}$$
\end{prop}

\begin{proof}
We have $\F_i^{[j]} \subset \E_i [\pi^j] = \pi^{e-j} \cdot \E_i$. Since $V_i$ sends $\E_i$ into $(\F_{i-1}^{[e]})^{(p)}$, it sends $\F_i^{[j]}$ into
$$\pi^{e-j} \cdot (\F_{i-1}^{[e]})^{(p)} \subset (\F_{i-1}^{[j]})^{(p)}$$
\end{proof}

\noindent The Verschiebung then induces maps
$$Ha_i^{[j]} (G) : \F_i^{[j]} / \F_i^{[j-1]} \to (\F_{i-1}^{[j]} / \F_{i-1}^{[j-1]})^{(p)}$$
for $1 \leq i \leq f$ and $1 \leq j \leq e$. The following proposition was proved in \cite{R-X} in the case of the Hilbert modular variety.

\begin{prop} \label{prod}
Let $1 \leq i \leq f$, and $1 \leq j \leq e$. We have the equality
$$Ha_i^{[j]} (G) =  (M_{i-1}^{[j+1]})^{(p)}  \circ \dots \circ (M_{i-1}^{[e]})^{(p)} \circ Hasse_i (G) \circ M_i^{[2]} (G)  \dots  \circ M_i^{[j]} (G)$$
\end{prop}

\begin{proof}
The map on the right-hand side is the composition of the multiplication by $\pi^{j-1}$, the map $Hasse_i(G)$, and the multiplication by $\pi^{e-j}$. Since $Hasse_i(G)$ is the composition of the division by $\pi^{e-1}$ and the Verschiebung, the result follows.
\end{proof}

\noindent Let $1 \leq i \leq f$, and $1 \leq j \leq e$. Taking the determinant of the previous map, one has a section $ha_i^{[j]} (G) \in H^0(S, (\Li_{G,i-1}^{[j]})^p (\Li_{G,i}^{[j]})^{-1})$. These sections are called the partial Hasse invariants. Similarly, one has a map induced by the Frobenius
$$Ha_i^{[j]} (G^D)^\vee : (\F_{i-1}^{[2e+1-j]} / \F_{i-1}^{[2e-j]})^{(p)} \to \F_{i}^{[2e+1-j]} / \F_{i}^{[2e-j]}$$
and a section $ha_i^{[j]} (G^D) \in H^0(S, (\Li_{G^D,i-1}^{[j]})^p (\Li_{G^D,i}^{[j]})^{-1})$.

\begin{theo}
Let $1 \leq i \leq f$, and $1 \leq j \leq e$. There is an isomorphism $(\Li_{G,i-1}^{[j]})^p (\Li_{G,i}^{[j]})^{-1} \simeq~(\Li_{G^D,i-1}^{[j]})^p (\Li_{G^D,i}^{[j]})^{-1}$. With this isomorphism, we have $ha_i^{[j]} (G) = ha_i^{[j]} (G^D)$.
\end{theo}

\begin{proof}
We have the following equality
$$(\Li_{G,i-1}^{[j]})^p (\Li_{G,i}^{[j]})^{-1} = \prod_{k=0}^{e-j-1} \left(\Li_{G,i-1}^{[j+k]} (\Li_{G,i-1}^{[j+k+1]})^{-1}\right)^p  (\Li_{G,i-1}^{[e]})^p (\Li_{G,i}^{[1]})^{-1} \prod_{k=1}^{j-1} \left(\Li_{G,i}^{[k]} (\Li_{G,i}^{[k+1]})^{-1}\right) $$
The compatibility with the duality has already been proved for the invertible sheaves $\Li_{G,k}^{[l-1]} (\Li_{G,k}^{[l]})^{-1}$ and $(\Li_{G,k-1}^{[e]})^p (\Li_{G,k}^{[1]})^{-1}$ for $1 \leq k \leq f$ and $2 \leq l \leq e$, hence the result for the sheaf $(\Li_{G,i-1}^{[j]})^p (\Li_{G,i}^{[j]})^{-1} $. \\
The proposition $\ref{prod}$ proves that the section $ha_i^{[j]} (G)$ can be expressed as a product of the sections $m_k^{[l]} (G)$ and $hasse_k (G)$, with $1 \leq k \leq f$ and $2 \leq l \leq e$. Since the compatibility with the duality has been proved for these sections, the result follows.

\end{proof}

\bibliographystyle{amsalpha}

\end{document}